\documentclass[authoryear]{elsarticle}

\journal{Statistics \& Probability Letters}

\usepackage{amsmath,amssymb,amsthm}

\newcommand{\eps}{\varepsilon}
\newcommand{\E}{\mathsf E}
\newcommand{\Prob}{\mathsf P}
\newcommand{\R}{\mathbb R}

\DeclareMathOperator{\sign}{sign}

\newtheorem{theorem}{Theorem}[section]
\newtheorem{proposition}[theorem]{Proposition}
\newtheorem{corollary}[theorem]{Corollary}
\theoremstyle{definition} \newtheorem{definition}[theorem]{Definition}
\theoremstyle{remark}
\newtheorem{remark}[theorem]{Remark}

\begin{document}

\begin{frontmatter}

\title{Multifractional Poisson process, multistable subordinator and related limit theorems}

\author[mainaddress]{Ilya Molchanov}
\ead{ilya.molchanov@stat.unibe.ch}

\author[mainaddress,secondaryaddress]{Kostiantyn Ralchenko\corref{mycorrespondingauthor}}
\cortext[mycorrespondingauthor]{Corresponding author}
\ead{k.ralchenko@gmail.com}

\address[mainaddress]{University of Bern,
 Institute of Mathematical Statistics and Actuarial Science,
 Sidlerstrasse 5,
 CH-3012 Bern,
 Switzerland}

\address[secondaryaddress]{Taras Shevchenko National University of Kyiv,
 Department of Probability Theory, Statistics and Actuarial Mathematics,
 Volodymyrska 64/13,
 01601 Kyiv,
 Ukraine}

\begin{abstract}
We introduce a multistable subordinator, which generalizes the stable subordinator to the case of time-varying stability index.
This enables us to define a multifractional Poisson process.
We study properties of these processes and establish the convergence of a continuous-time random walk to the multifractional Poisson process.
\end{abstract}


\end{frontmatter}

\section{Introduction}
The multistable processes are generalizations of well-known stable processes.
They are locally stable, but the stability index may evolve over time.
Various definitions of such processes can be found in~\cite{FalconerLiu12,Falconer_LeGuevel_LevyVehel_2009,Falconer_LevyVehel_2009,LeGuevel_LevyVehel_2012,LeGuevelLevyVehelLiu13}.

The literature contains several definitions of the fractional Poisson process, see~\cite{Beghin_Orsingher09,Laskin03,Mainardi_Gorenflo_Scalas04,RepinSaichev_2000,Uchaikin_Cahoy_Sibatov08}.
\cite{MeerschaertNaneVellaisamy11} suggested to define a fractional Poisson process as $N(E(t))$, for a Poisson process $N(t)$ and $E(t)$ being the right-continuous inverse of the standard $\beta$-stable subordinator $D(t)$, independent of $N(t)$.

In this paper, we construct a multifractional Poisson process using the multistable subordinator instead of $D(t)$.
This process has non-stationary increments, but it retains some important properties of the usual $\beta$-stable subordinator such as the stochastic continuity, the independence of the increments and the strictly increasing sample paths.
We establish that the multistable subordinator is a weak limit of the sums
$\sum_{k\le nt}b_{nk}^{-1}J_{nk}$
where
$\{J_{nk},n\ge1,k\ge1\}$
are independent random variables with regularly varying tails and
$\{b_{nk},n\ge1,k\ge1\}$
are some normalizing constants.
Moreover, we prove that the corresponding right-continuous inverse converges to the right-continuous inverse of the multistable subordinator.
This makes possible to construct continuous-time random walks, which converge to the multifractional Poisson process.
These results generalize
\cite[Theorem~2.5]{MeerschaertNaneVellaisamy11} and \cite[Theorem~3.2]{MeerschaertScheffler04} to the multistable case.

The paper is organized as follows.
In Section~\ref{sec:2} we introduce multistable subordinator, its inverse and multifractional Poisson process, and investigate their properties.
In Section~\ref{sec:3} we prove various limit theorems for these processes.

\section{Multistable subordinator and multifractional Poisson process}\label{sec:2}
\subsection{Basic definitions}
Let $\beta\colon\R_+\to(0,1)$ be a continuous function, $\R_+=[0,\infty)$.
Denote
\[
\beta^*=\sup_t\beta(t),\quad
\beta_*=\inf_t\beta(t).
\]
Consider a Poisson point process $\Pi=\{(t_i,x_i)\}$ on $\R_+\times(0,\infty)$ with intensity measure
\[
\nu(dt,dx)=\beta(t)x^{-\beta(t)-1} dt\,dx.
\]
Define
\[
D(t)=\sum_{\substack{t_i\le t\\(t_i,x_i)\in\Pi}}x_i,
\quad t\ge0.
\]
Since
\begin{equation}\label{eq:exist}
\int_{[0,t]\times(0,\infty)}\min(1,x)\nu(ds,dx)
=\int_0^t\frac{ds}{1-\beta(s)},
\end{equation}
the process $D(t)$ is well defined, if the function
$(1-\beta(s))^{-1}$ is integrable.  In particular, this condition
holds if $\beta^*<1$.  It follows from \cite{FiszVaradarajan} that the
distribution of $D(t)$ is absolutely continuous for all $t>0$.

The Laplace transform of $D(t)$ is equal to
\begin{equation}\label{eq:Laplace}
\begin{split}
\E\exp\{-\theta D(t)\}
&=\exp\left\{\int_0^t\int_0^\infty\left(e^{-\theta x}-1\right)\beta(s)x^{-\beta(s)-1}dx\,ds\right\}\\
&=\exp\left\{-\int_0^t\Gamma(1-\beta(s))\theta^{\beta(s)}ds\right\},
\end{split}
\end{equation}
see, for instance, \cite[Section 3.2]{Kingman93}.
If $\beta(t)=\beta$ is a constant,
\[
\E\exp\{-\theta D(t)\}
=\exp\left\{-\Gamma(1-\beta)t\theta^{\beta}\right\},
\]
is the Laplace transform of the $\beta$-stable subordinator.
We will call the process $D(t)$ the \emph{multistable subordinator} with index $\beta(t)$, $t\ge0$.

\begin{remark}
It is not hard to see that $D(t)$ can be represented as a sum over the stationary Poisson point process $\Pi'$ on $\R_+\times(0,\infty)$ (with the Lebesgue intensity measure),
namely
\[
D(t)\overset{d}{=}\sum_{\substack{t_i\le t\\(t_i,x_i)\in\Pi'}}x_i^{-1/\beta(t_i)}.
\]
\end{remark}

\begin{remark}
\cite{LeGuevelLevyVehelLiu13} study the so-called \emph{independent increments multistable L\'evy motion}.
It admits the following representation \cite[Proposition 1]{LeGuevelLevyVehelLiu13}:
\[
L(t)=\sum_{\substack{t_i\le t\\(t_i,x_i)\in\Pi''}}C_{\beta(t_i)}^{1/\beta(t_i)}x_i^{<-1/\beta(t_i)>}, \quad t\in\R_+,
\]
where
$\Pi''$ is a stationary Poisson point process on $\R_+\times\R$,
$C_\beta=\Gamma(1-\beta)\cos\frac{\beta\pi}{2}$
and
$x_i^{<-1/\beta(t_i)>}=\sign(x_i)|x_i|^{-1/\beta(t_i)}$.
We construct a multistable process using a similar approach, but we avoid using factor $C_\beta$ and consider only positive $x$'s in order to obtain a process with positive and strictly increasing sample paths. Note that the sample paths of $L(t)$ are piecewise-constant functions, see \citep{LeGuevelLevyVehelLiu13}.
\end{remark}

Let $E(r)$ be the right-continuous inverse of $D(t)$, i.\,e.
\[
E(r)=\inf\{t\ge0:D(t)\ge r\}.
\]
Furthermore, let $N(t)$ be a homogeneous Poisson process with intensity $\lambda$, independent of $\Pi$.
\begin{definition}
We call the process $X(t)=N(E(t))$, $t\ge0$, \emph{multifractional Poisson process} with index $\beta(t)$, $t\ge0$.
\end{definition}

\subsection{Properties of the multistable subordinator}
In this subsection we will prove some useful properties of the processes $D(t)$ and $E(t)$.

\begin{proposition}
Assume that $\beta^*<1$.
Then the process $D$ is continuous in probability.
Moreover, for any $\eps>0$ there exists a constant $C_\eps>0$ such that
\begin{equation}\label{eq:up_bound}
\Prob(D(t+h)-D(t)>\eps)\le C_\eps h
\end{equation}
for all $t,h\in\R_+$.
\end{proposition}

\begin{proof}
We have
\begin{multline}\label{eq:ineq1}
\Prob(D(t+h)-D(t)>\eps)\\
\le
\Prob\left(\sum_{\substack{t<t_i\le t+h, x_i\le1\\(t_i,x_i)\in\Pi}}x_i>\frac\eps2\right)
+\Prob\left(\sum_{\substack{t<t_i\le t+h, x_i>1\\(t_i,x_i)\in\Pi}}x_i>\frac\eps2\right).
\end{multline}
By Markov's inequality, we obtain
\begin{multline*}
\Prob\left(\sum_{\substack{t<t_i\le t+h, x_i\le1\\(t_i,x_i)\in\Pi}}x_i>\frac\eps2\right)
\le\frac2\eps\E\sum_{\substack{t<t_i\le t+h, x_i\le1\\(t_i,x_i)\in\Pi}}x_i\\
=\frac2\eps\int_t^{t+h}\int_0^1\beta(s)x^{-\beta(s)}\,dx\,ds
=\frac2\eps\int_t^{t+h}\frac{\beta(s)}{1-\beta(s)}\,ds
\le\frac{2\beta^*}{\eps(1-\beta^*)}h.
\end{multline*}
The inequality under the ultimate probability in~\eqref{eq:ineq1}
implies that the Poisson process $\Pi$
has at least one point in
$(t,t+h]\times(1,\infty)$.
Therefore,
\begin{multline*}
\Prob\left(\sum_{\substack{t<t_i\le t+h, x_i>1\\(t_i,x_i)\in\Pi}}x_i>\frac\eps2\right)\\
\le1-\exp\left\{-\int_{t}^{t+h}\!\!\int_1^\infty
\beta(s)x^{-\beta(s)-1}ds\,dx\right\}
=1-e^{-h}\le h.
\end{multline*}
Thus, \eqref{eq:up_bound} holds with
$C_\eps=1+\frac{2\beta^*}{\eps(1-\beta^*)}$.
\end{proof}

\begin{proposition}
The process $D(t)$, $t\ge0$ has independent increments.
\end{proposition}

\begin{proof}
The result follows from the Poisson property of $\Pi$.
\end{proof}

\begin{proposition}\label{prop:D-increasing}
The sample paths of $D$ are a.\,s.\ strictly increasing.
\end{proposition}

\begin{proof}
For all $t\ge0$ and $h>0$, the event
$\{D(t+h)=D(t)\}$ means that the Poisson process $\Pi$ has no points in $(t,t+h]\times(0,\infty)$.
Therefore,
$\Prob(D(t+h)=D(t))=0$.
Then $D$ is strictly increasing by considering rational $t$ and $h$.
\end{proof}

\begin{corollary}\label{cor:E-cont}
The sample paths of $E$ are a.\,s.\ continuous and non-decreasing.
\end{corollary}

\begin{proof}
The statement follows directly from the definition of $E$ and Proposition~\ref{prop:D-increasing}.
\end{proof}

\section{Limit theorems}\label{sec:3}
\subsection{Convergence of point processes in the scheme of series}
Let $\{J_{nk},n\ge1,k\ge1\}$ be non-negative independent random variables such that
\begin{equation}\label{eq:J_nk}
\Prob(J_{nk}>t)=t^{-\beta(k/n)}L\left(t^{\beta(k/n)/\beta^*}\right)
\end{equation}
for some slowly varying function $L(t)$, i.\,e.
\begin{equation}\label{eq:slow_var}
L(\lambda t)/L(t)\to1\quad
\text{as } t\to\infty \quad \text{for any } \lambda>0.
\end{equation}
Define the sequences of normalizing constants
$\{a_n,n\ge1\}$ and $\{b_{nk},n\ge1,k\ge1\}$
by
\begin{gather}
a_n^{-\beta^*}L(a_n)=\frac1n,
\label{eq:b_n*}\\
b_{nk}=a_n^{\beta^*/\beta(k/n)}.
\label{eq:b_nk}
\end{gather}

Let $\delta_x$ be the delta-measure concentrated at $x$.

\begin{theorem}
The point process
$N_n=\sum_{k=1}^\infty\delta_{(kn^{-1},b_{nk}^{-1}J_{nk})}$
weakly converges to the Poisson point process $\Pi$.
\end{theorem}
\begin{proof}
The proof follows the scheme from \cite[Proposition~3.21]{Resnick87}.
According to~\cite[Proposition~3.19]{Resnick87}, it suffices to prove the convergence of Laplace functionals
$\psi_{N_n}(f)\to\psi_{\Pi}(f)$
for arbitrary function $f$ from the family $C^+_0(\R_+\times(0,\infty))$ of continuous, non-negative functions on\linebreak $\R_+\times(0,\infty)$ with compact support.
We have
\[
\psi_{\Pi}(f)=\exp\left\{-\int_{\R_+\times(0,\infty)}\left(1-e^{-f(s,x)}\right)\nu(ds,dx)\right\},
\]
see, for example, \cite[Proposition~3.6]{Resnick87}, and
\begin{equation}\label{eq:psi_xi_n}
\begin{split}
\psi_{N_n}(f)
&=\E\exp\left\{-\sum_k f(kn^{-1},b_{nk}^{-1}J_{nk})\right\}\\
&=\prod_k\left(1-\int_0^\infty\left(1-e^{-f(kn^{-1},b_{nk}^{-1}y)}\right)
    \Prob(J_{nk}\in dy)\right)\\
&=\prod_k\left(1-\int_0^\infty\left(1-e^{-f(kn^{-1},x)}\right)
    \Prob(b_{nk}^{-1}J_{nk}\in dx)\right).
\end{split}
\end{equation}

Consider the measure
$\mu_{nk}(\cdot)= \Prob(b_{nk}^{-1}J_{nk}\in\cdot)$.
We claim that
\begin{equation}\label{eq:mu_nk}
\lim_{n\to\infty}\sup_k\left|n\mu_{nk}([c_1,c_2])-\left(c_1^{-\beta(k/n)}-c_2^{-\beta(k/n)}\right)\right|=0
\end{equation}
for any $[c_1,c_2]\subset(0,\infty)$.
Indeed, using the relations \eqref{eq:J_nk}, \eqref{eq:b_n*} and \eqref{eq:b_nk}, we can write
\begin{equation}\label{eq:mu_nk2}
\begin{split}
n\mu_{nk}([c_1,c_2])
&=n(\Prob(J_{nk}>b_{nk}c_1)-\Prob(J_{nk}>b_{nk}c_2))\\
&=\frac{(b_{nk}c_1)^{-\beta(k/n)}L\left((b_{nk}c_1)^{\frac{\beta(k/n)}{\beta^*}}\right)
    -(b_{nk}c_2)^{-\beta(k/n)}L\left((b_{nk}c_2)^{\frac{\beta(k/n)}{\beta^*}}\right)}%
    {a_n^{-\beta^*}L(a_n)}\\
&=c_1^{-\beta(k/n)}\frac{L\left(a_nc_1^{\frac{\beta(k/n)}{\beta^*}}\right)}{L(a_n)}
    -c_2^{-\beta(k/n)}\frac{L\left(a_nc_2^{\frac{\beta(k/n)}{\beta^*}}\right)}{L(a_n)}.
\end{split}
\end{equation}
It is known (see \citealp[Theorem 1.1]{Seneta76}) that for every fixed $[d_1,d_2]\subset(0,\infty)$ the convergence~\eqref{eq:slow_var} holds uniformly with respect to $\lambda\in[d_1,d_2]$.
Therefore, \eqref{eq:mu_nk} follows from~\eqref{eq:mu_nk2}, since
$a_n\to\infty$ as $n\to\infty$.

Define $\nu_n$ by
\[
\nu_n(ds,dx)=\sum_k\delta_{kn^{-1}}(ds)\mu_{nk}(dx).
\]
Then
\begin{multline*}
\nu_n((u,v]\times(c_1,c_2])\\
=\frac1n\sum_k\delta_{kn^{-1}}((u,v])\left(n\mu_{nk}((c_1,c_2])-\left(c_1^{-\beta(kn^{-1})}-c_2^{-\beta(kn^{-1})}\right)\right)\\
+\int_{c_1}^{c_2}\left(\frac1n\sum_k\delta_{kn^{-1}}((u,v])\beta(kn^{-1})x^{-\beta(kn^{-1})-1}\right)dx
\end{multline*}
The first term in the right-hand side converges to 0 as $n\to\infty$ by~\eqref{eq:mu_nk}.
The second one converges to
$\nu((u,v]\times(c_1,c_2])$,
since the integrand is the Riemann sum for $\int_u^v\beta(s)x^{-\beta(s)-1}ds$.
Hence, $\nu_n$ converges weakly to $\nu$.

Therefore,
\begin{multline}\label{eq:conv_to_loq_psi_Pi}
\sum_k\int_0^\infty\left(1-e^{-f(kn^{-1},x)}\right)\mu_{nk}(dx)
=\int_{\R_+\times(0,\infty)}\left(1-e^{-f}\right)d\nu_{n}\\
\to\int_{\R_+\times(0,\infty)}\left(1-e^{-f(s,x)}\right)\nu(ds,dx)
=-\log\psi_{\Pi}(f)
\end{multline}
as $n\to\infty$.

If $K\subset\R_+\times(0,\infty)$ is the compact support of $f$ and $[c_1,c_2]\subset(0,\infty)$ is its projection on the second coordinate, then
\begin{equation}\label{eq:sup_to0}
\sup_k\int_0^\infty\left(1-e^{-f(kn^{-1},x)}\right)\mu_{nk}(dx)
\le\sup_k\mu_{nk}([c_1,c_2])
\to0 \qquad\text{as }n\to\infty
\end{equation}
by \eqref{eq:mu_nk}.
Equation~\eqref{eq:psi_xi_n} implies
\[
-\log\psi_{N_n}(f)
=-\sum_k\log\left(1-\int_0^\infty\left(1-e^{-f(kn^{-1},x)}\right)\mu_{nk}(dx)\right).
\]
The elementary expansion
\[
\log(1+z)=z(1+\varepsilon(z))
\]
for
$|\varepsilon(z)|\le|z|$
if $|z|\le1/2$ yields
\begin{multline*}
\left|-\log\psi_{N_n}(f)-\sum_k\int_0^\infty\left(1-e^{-f(kn^{-1},x)}\right)\mu_{nk}(dx)\right|\\
\le\sum_k\left(\int_0^\infty\left(1-e^{-f(kn^{-1},x)}\right)\mu_{nk}(dx)\right)^2\\
\le\left(\sup_k\int_0^\infty\left(1-e^{-f(kn^{-1},x)}\right)\mu_{nk}(dx)\right)\\
\times\sum_k\int_0^\infty\left(1-e^{-f(kn^{-1},x)}\right)\mu_{nk}(dx)\to0
\end{multline*}
as $n\to\infty$ by \eqref{eq:conv_to_loq_psi_Pi} and \eqref{eq:sup_to0}.
Together with~\eqref{eq:conv_to_loq_psi_Pi} this implies that
\[
-\log\psi_{N_n}(f)\to-\log\psi_{\Pi}(f)
\qquad\text{as } n\to\infty.
\]
Hence,
$\psi_{N_n}(f)\to\psi_{\Pi}(f)$ as $n\to\infty$.
\end{proof}

\subsection{Convergence of sums to the multistable subordinator}
\begin{theorem}\label{th2}
Assume that
$\beta_*>0$ and $\beta^*<1$.
Then
$\sum_{k\le nt}b_{nk}^{-1}J_{nk}$
weakly converges to $D(t)$ as $n\to\infty$.
\end{theorem}

\begin{proof}
By the Skorohod theorem, it is possible to define $N_n$ and $\Pi$ on the same probability space so that $N_n\to\Pi$ a.\,s\ as $n\to\infty$.
For all $r\in(0,1)$ we have
\begin{multline*}
\left|\sum_{k\le nt}b_{nk}^{-1}J_{nk} - D(t)\right|
\le\left|\sum_{\substack{k\le nt\\b_{nk}^{-1}J_{nk}\le r}}
b_{nk}^{-1}J_{nk}\right|\\
+\left|\sum_{\substack{k\le nt\\b_{nk}^{-1}J_{nk}\ge r}}
b_{nk}^{-1}J_{nk}-\sum_{\substack{t_i\le t, x_i\ge r\\(t_i,x_i)\in\Pi}}x_i\right|
+\left|\sum_{\substack{t_i\le t, x_i\le r\\(t_i,x_i)\in\Pi}}x_i\right|\\
=\eta_n(r)+\zeta_n(r)+\gamma(r).
\end{multline*}

Note that the processes $N_n$ and $\Pi$ have a.\,s.\ at most finite number of points in $[0,t]\times[r,\infty)$ for every $t>0$ and $r>0$.
Therefore the a.\,s.\ convergence $N_n\to\Pi$ implies that the second summand
$\zeta_n(r)\to0$
a.\,s.\ as $n\to\infty$ for all $r>0$.

The third summand
$\gamma(r)=\int_{[0,t]\times(0,r]}x\,d\Pi\to0$
a.\,s.\ as $r\downarrow0$, since the integral
$\int_{[0,t]\times(0,1]}x\,d\Pi$
is a.\,s.\ finite.
Indeed,
$\int_{[0,t]\times(0,1]}x\,\nu(ds,dx)<\infty$,
see~\eqref{eq:exist}.

Consider the first summand.
If $r<1$, then
\begin{align*}
\eta_n(r)&=\sum_{\substack{k\le nt\\b_{nk}^{-1}J_{nk}\le r}}
b_{nk}^{-1}J_{nk}
\le\sum_{\substack{k\le nt\\b_{nk}^{-1}J_{nk}\le r}}
\left(b_{nk}^{-1}J_{nk}\right)^{\beta(kn^{-1})/\beta^*}\\
&=a_n^{-1}\sum_{\substack{k\le nt\\b_{nk}^{-1}J_{nk}\le r}}
J_{nk}^{\beta(kn^{-1})/\beta^*},
\end{align*}
where the last equality follows from~\eqref{eq:b_nk}.
Similarly, we can rewrite the inequality
$b_{nk}^{-1}J_{nk}\le r$
in the following form
\[
a_n^{-1}J_{nk}^{\beta(kn^{-1})/\beta^*}\le r^{\beta(kn^{-1})/\beta^*}.
\]
Then
\begin{equation}\label{eq:bound}
\eta_n(r)\le\sum_{\substack{k\le nt\\a_n^{-1}J_{nk}^{\beta(kn^{-1})/\beta^*}\le r^{\beta_*/\beta^*}}}
a_n^{-1}J_{nk}^{\beta(kn^{-1})/\beta^*}
\end{equation}
for any $r\in(0,1)$.

It follows from~\eqref{eq:J_nk} that
\[
\Prob\left(J_{nk}^{\beta(kn^{-1})/\beta^*}>t\right)=t^{-\beta^*}L(t).
\]
This means that
$\left\{J_{nk}^{\beta(kn^{-1})/\beta^*}, n\ge1, k\ge1\right\}$ is a sequence of i.\,i.\,d.\ random variables with regularly varying tails.
Binomial point processes generated by i.\,i.\,d.\ sequence of random variables with such tails were considered in \cite{DavydovMolchanovZuyev08} and the following property was established: if $\xi_1,\xi_2,\ldots$ are i.i.d. random variables with
$\Prob(\xi_1>t)=t^{-\beta^*}L(t)$
and $\{a_n, n\ge1\}$ is a corresponding sequence of normalizing constants, then the process
$\beta_n=\sum_{k=1}^n\delta_{a_n^{-1}\xi_k}$
satisfies the condition
\[
\limsup_n\Prob\left(\left|\int_{[0,r]}x\beta_n(dx)\right|\ge\varepsilon\right)\to0
\quad\text{as }r\downarrow0
\]
for each $\varepsilon>0$ (see the proof of Theorem~4.6 in \citealp{DavydovMolchanovZuyev08}).
Using this fact, we deduce from~\eqref{eq:bound} that
\[
\limsup_n\Prob(\eta_n(r)\ge\varepsilon)\to0
\quad\text{as }r\downarrow0.
\]

Thus,
\begin{multline*}
\Prob\left(\left|\sum_{k\le nt}b_{nk}^{-1}J_{nk} - D(t)\right|\ge\varepsilon\right)
\le\Prob(\eta_n(r)+\zeta_n(r)+\gamma(r)\ge\varepsilon)\\
\le\Prob(\eta_n(r)\ge\varepsilon/2)
+\Prob(\zeta_n(r)+\gamma(r)\ge\varepsilon/2).
\end{multline*}
This probability can be made arbitrary small by the choice of $n$ and $r$.
\end{proof}

\begin{corollary}\label{cor:conv_fd}
Assume that
$\beta_*>0$ and $\beta^*<1$.
Then the finite-dimensional distributions of
$E_n(r)=\inf\left\{t:\sum_{k\le nt}b_{nk}^{-1}J_{nk}\ge r\right\}$ converge to those of $E(r)$
as $n\to\infty$.
\end{corollary}
\begin{proof}
For arbitrary $m$, $r_1,\ldots,r_m$ and $u_1,\ldots,u_m$
\begin{align*}
&\Prob\left(\inf\left\{t:\sum_{k\le nt}b_{nk}^{-1}J_{nk}\ge r_j\right\}\le u_j, j=1,\ldots,m\right)\\
&\quad=\Prob\left(\sum_{k\le nu_j}b_{nk}^{-1}J_{nk}\ge r_j, j=1,\ldots,m\right)
\to\Prob\left(D(u_j)\ge r_j, j=1,\ldots,m\right)\\
&\quad=\Prob\left(\inf\left\{t:D(t)\ge r_j\right\}\le u_j, j=1,\ldots,m\right)\\
&\quad=\Prob\left(E(r_j)\le u_j, j=1,\ldots,m\right)
\end{align*}
as $n\to\infty$ by Theorem~\ref{th2}.
\end{proof}

\begin{corollary}\label{cor:conv_J1}
Assume that
$\beta_*>0$ and $\beta^*<1$.
Then $E_n$ weakly converges to $E$ in $D(\R_+,\R_+)$ as $n\to\infty$ in the $J_1$ topology.
\end{corollary}

\begin{proof}
Since $E_n$ has non-decreasing sample paths and $E$ is continuous in probability by Corollary~\ref{cor:E-cont},
the convergence under $J_1$ follows from Corollary~\ref{cor:conv_fd} and \cite[Theorem 3]{Bingham71}.
\end{proof}

\subsection{Convergence of continuous-time random walks to the multifractional Poisson process}
Let $Y_i^{(p)}$, $i\ge1$, be i.\,i.\,d.\ random variables, independent of $J_{nk}$, $n\ge1$, $k\ge1$,
such that
\[
\Prob\left(Y_i^{(p)}=1\right)=p, \quad
\Prob\left(Y_i^{(p)}=0\right)=1-p.
\]
Denote
$S^{(p)}(t)=\sum_{i\le t}Y_i^{(p)}$.

\begin{theorem}
Assume that $\beta_*>0$ and $\beta^*<1$.
If $p_n\downarrow0$ as $n\to\infty$,
then the continuous-time random walk
$S^{(p_n)}(\lambda E_n(t)/p_n)$
converges to the multifractional Poisson process $N(E(t))$ as $n\to\infty$ in the $M_1$ topology on $D(\R_+,\R)$.
\end{theorem}

\begin{proof}
It was established in the proof of~\cite[Theorem 2.5]{MeerschaertNaneVellaisamy11} that
$S^{(p)}(\lambda t/p)$ weakly converges to $N(t)$
as $p\to0$ in the $J_1$ topology.
By Corollary~\ref{cor:conv_J1} and~\cite[Theorem 3.2]{Billingsley68},
$\left(S^{(p_n)}(\lambda t/p_n),E_n(t)\right)$
weakly converges to $(N(t),E(t))$ as $n\to\infty$ in the $J_1$ topology of the product space $D(\R_+,\R\times\R)$.
In order to prove the convergence of the compositions in the $M_1$ topology, we will apply~\cite[Theorem 13.2.4]{Whitt02}.
Taking into account the continuity of $E$ (Corollary \ref{cor:E-cont}), it is sufficient to check that $t=E(r)$ is (almost surely) strictly increasing at $r$ whenever
$N(t-)\ne N(t)$.
It is easy to see that this condition holds if and only if the processes $D$ and $N$ have (almost surely) no simultaneous jumps.
Thus, we need to verify that the probability that both $D(t)-D(t-)$ and $N(t)-N(t-)$ exceed $\eps$ for some $t\ge0$
is equal to zero for arbitrary $\eps>0$.
Obviously, it is sufficient to prove this for a finite interval.
For the sake of simplicity, consider the interval $[0,1]$.
Then
\begin{multline*}
\Prob(D(t)-D(t-)>\eps,N(t)-N(t-)>\eps \text{ for some } t\in[0,1])\\
\le\Prob\left(\bigcap_{j=1}^\infty A_{2^j}\right)
=\lim_{j\to\infty}\Prob(A_{2^j}),
\end{multline*}
where
$
A_j=\bigcup_{k=0}^{j-1}\left\{D\left(\frac{k+1}{j}\right)-D\left(\frac{k}{j}\right)>\eps, N\left(\frac{k+1}{j}\right)-N\left(\frac{k}{j}\right)>\eps\right\}.
$
Using the independence of processes, the estimate~\eqref{eq:up_bound} for $D$ and the stationarity of the increments for $N$, we obtain
\begin{align*}
\Prob(A_j)&\le\sum_{k=0}^{j-1}\Prob\left(D\left(\tfrac{k+1}{j}\right)-D\left(\tfrac{k}{j}\right)>\eps\right) \Prob\left(N\left(\tfrac{k+1}{j}\right)-N\left(\tfrac{k}{j}\right)>\eps\right)\\
&\le\frac{C_\eps}{j}\sum_{k=0}^{j-1}\Prob\left(N\left(\tfrac{k+1}{j}\right)-N\left(\tfrac{k}{j}\right)>\eps\right)
=C_\eps\Prob(N(1/j)>\eps).
\end{align*}
Since the Poisson process $N$ is continuous in probability, we have
$\Prob(A_j)\to0$ as $j\to\infty$.
This completes the proof.
\end{proof}

\section*{Acknowledgements}
IM supported in part by Swiss National Science Foundation grant 200021-137527.
KR supported by the Swiss Government Excellence Scholarship.

The authors are grateful to the anonymous referees for very careful reading of this paper and suggesting a number of improvements.

\bibliographystyle{elsarticle-harv}
\bibliography{ctrw}

\end{document}